\documentclass[10pt]{amsart}
\usepackage[utf8]{inputenc}
\usepackage{color}

\usepackage{amsthm}
\usepackage{amssymb}
\usepackage{amsmath}
\usepackage{bbm}
\usepackage{pdfpages}
\usepackage{booktabs}
\usepackage{graphics}
\usepackage{tikz,lipsum,lmodern}
\usepackage{tikz-cd}
\usepackage{sidecap}
\usepackage{subcaption}
\usepackage[linesnumbered,ruled,vlined]{algorithm2e}

\sidecaptionvpos{figure}{c}
\usepackage[most]{tcolorbox}
\colorlet{shadecolor}{gray!40}
\usepackage[bookmarksnumbered=true]{hyperref} 
\hypersetup{
     colorlinks = true,
     linkcolor = blue,
     anchorcolor = blue,
     citecolor = blue,
     filecolor = blue,
     urlcolor = blue
     }
\usetikzlibrary{patterns}

\theoremstyle{plain}

\newtheorem{theorem}{Theorem}[section]

\newtheorem{proposition}[theorem]{Proposition}

\theoremstyle{definition}

\newtheorem{definition}[theorem]{Definition}
\newtheorem{example}[theorem]{Example}

\theoremstyle{remark}

\newtheorem{remark}[theorem]{Remark}

\def\NN{{\mathbb N}}
\def\ZZ{{\mathbb Z}}
\def\kb{{\mathbf k}}
\def\xb{{\mathbf x}}
\def\ab{{\mathbf a}}
\def\Pc{{\mathcal P}}
\def\Fc{{\mathcal F}}
\def\Hc{{\mathcal H}}
\def\Dc{{\mathcal D}}
\def\mk{{\mathfrak m}}

\def\lcm{\mathrm{lcm}}
\def\supp{\mathrm{supp}}
\def\ms{\mathrm{ms}}
\def\height{\mathrm{height}}
\def\projdim{\mathrm{projdim}}
\def\reg{\mathrm{reg}}

\DeclareMathOperator{\Tor}{Tor}

\newcommand{\1}{\mathbbm{1}}

\title[Polarization and depolarization of abstract simplicial complexes]{Monomial polarization and depolarization of abstract simplicial complexes}

\author[L\'opez]{Víctor M. L\'opez-Antón}
\address{Departamento de Matemáticas y Computación, Universidad de La Rioja, Spain}
\email{victor-manuel.lopez@unirioja.es}

\author[Munarriz]{Pablo Munarriz-Senosiain}
\address{Departamento de Matemáticas y Computación, Universidad de La Rioja, Spain}
\email{pablo.munarriz@unirioja.es}

\author[Pascual]{Patricia Pascual-Ortigosa}
\address{Departamento de Matemáticas y Computación, Universidad de La Rioja, Spain}
\email{patricia.pascualo@unirioja.es}

\author[S\'aenz-de-Cabez\'on ]{Eduardo S\'aenz-de-Cabez\'on}
\address{Departamento de Matemáticas y Computación, Universidad de La Rioja, Spain}
\email{eduardo.saenz-de-cabezon@unirioja.es}

\begin{document}

\begin{abstract}
We translate the operations of polarization and depolarization from monomial ideals in a polynomial ring to abstract simplicial complexes. As a result, we explicitly describe the relation between the Koszul simplicial complex of a monomial ideal and that of its polarization. Using the simplicial translation of depolarization we propose a way to reduce a simplicial complex to a smaller one with the same homology. This type of reduction, that can be interpreted as non-elementary collapse, can be used as a pre-process step for algorithms on simplicial complexes. We apply this methodology to the efficient computation of the Alexander dual of abstract simplicial complexes.
\end{abstract}

\maketitle
\section{Introduction}
Polarization and depolarization are ways to study monomial ideals in a polynomial ring via squarefree monomial ideals and vice-versa. The operation of {\em polarization} assigns to every monomial ideal $I$ a squarefree one $\Pc(I)$ in another polynomial ring, in such a way that $I$ and $\Pc(I)$ have the same $\lcm$-lattice, Betti numbers and other important invariants \cite{V15,HH11,F05,GPW99}, see also \cite{AFL22,AV22,OY13} for a slightly different treatment of polarization. The opposite operation, {\em depolarization}, has been less studied \cite{MPSW20} and it is not unique, in the sense that the depolarizations of a squarefree ideal can be a set of several different monomial ideals. Squarefree monomial ideals, are in one-to-one correspondence to simplicial complexes via the Stanley-Reisner correspondence \cite{S75,S96,FMS14,F21}. This can be generalized to non-squarefree ideals by means of Koszul complexes \cite{MS05}. Our main goal is to express polarization and depolarization in the language of simplicial complexes in such a way that the Stanley-Reisner and Koszul correspondences remain coherent. Using this, we address three problems. First, we describe the relation between the Koszul simplicial complex of a monomial ideal and that of its polarization; second, we show how depolarization can reduce a simplicial complex while keeping its homology; and third, we use this reduction to efficiently compute the Alexander dual of abstract simplicial complexes. The rest of this introduction gives a more precise description of these three contributions.

Let $R=\kb[x_1,\dots,x_n]$ be a polynomial ring over a field $\kb$, and $I$ a monomial ideal in $R$. Let $\mu=(\mu_1,\dots,\mu_n)\in \NN^n$ be a multi-index, and $\xb^\mu=x_1^{\mu_1}\cdots x_n^{\mu_n}$ the corresponding monomial in $R$. The (upper) Koszul simplicial complex of $I$ at $\mu$ is defined in \cite{MS05} Ch. 1, as
\[
K^{\mu}_I=\left\{\sigma\subseteq\{1,\dots,n\}\mid \frac{\xb^\mu}{\xb_\sigma}\in I\right\}, \mbox{ where } \xb_\sigma=\prod_{i\in\sigma}x_i.
\]

If $I$ is a squarefree monomial ideal, then $K^{\1}_I=\Delta^\vee_I$, where $\1=(1,\dots,1)$ and $\Delta^\vee_I$ is the Alexander dual of the Stanley-Reisner complex of $I$, see \cite{S96}.

By the well-known Hochster's formula \cite{H77} and its variants, we know that for all $i\in\ZZ$ and $\mu\in\NN^n$
\begin{equation*}
    \beta_{i,\mu}(I)=\dim_\kb \widetilde{H}_{i-1}(K^\mu_I;\kb),
\end{equation*}
where $\beta_{i,\mu}(I)$
 is the $i$-th Betti number of $I$ at multidegree $\mu$, and $\widetilde{H}$ denotes reduced homology.

 The {\em polarization} of $I$, $\Pc(I)$, is an ideal in $n'$ variables, for a certain $n'\geq n$. We denote by $\Pc(R)$ the ring of $\Pc(I)$, and the polarization of $\mu\in\NN^n$ as $\Pc(\mu)$.
 It is known that Betti numbers are preserved under polarization \cite{HH11,GPW99}, i.e., $\beta_{i,\mu}(I)=\beta_{i,\Pc(\mu)}(\Pc(I)) \mbox{ for all } i\in\ZZ,\mu\in\NN^n$.

 Let $\mu_I=\lcm\{x^\nu\mid x^\nu\in G(I)\}$, where $G(I)$ is the (unique) minimal set of monomial generators of $I$. Then $\Pc(\mu_I)=\1_{\Pc(R)}$, and
 \[
 \dim_\kb\widetilde{H}_{i}(K^{\mu_I}_I;\kb)=\dim_\kb\widetilde{H}_{i}(K^\1_{\Pc(I)};\kb)=\dim_\kb\widetilde{H}_{i}(\Delta^\vee_{\Pc(I)};\kb).
 \]

 A question then arises:
 What is the topological/combinatorial relation between $K^{\mu_I}_I$ and $K^\1_{\Pc(I)}$? In this paper we give an answer to this question, showing explicitly how the polarization operation on $I$ transforms the simplicial complex $K^{\mu_I}_I$ on $K^\1_{\Pc(I)}$ while keeping the homology. This is our first main result, stated in Proposition \ref{prop:expandedIsKoszul}.
 
 We then explore the inverse procedure, i.e., given a simplicial complex, use {\em depolarization} on it. In this case, the simplicial version of depolarization can be seen as a way to reduce a simplicial complex $\Delta$ to a simpler or smaller one, $\Delta'$ with the same homology, in the vein of other well known strategies such as collapses \cite{W39,BG12,F24} or Discrete Morse Theory \cite{BV02,S06,F25}. The explicit description of this reduction is our second main result, Proposition \ref{prop:simplicialDepolarization}.

 Finally, using these constructions, we describe an algorithm to compute the Alexander dual of an abstract simplicial complex \cite{BT09}. The theoretical complexity of this problem is quasi-polynomial on the sum of the sizes of the input and output, and the size of the output can be exponential in terms of the input \cite{EGM02,RR25}. We propose Algorithm \ref{alg:AlexanderDual} which, given a simplicial complex $\Delta$, depolarizes its corresponding monomial ideal to obtain a simpler one, and computes the Alexander dual ideal of this depolarization. Finally, this dual is transformed so that the Alexander dual of $\Delta$ is obtained. In this procedure, the critical step, computing the Alexander dual, is performed on a smaller object, obtained by depolarization.

 The outline of the paper is as follows. In Section \ref{sec:preliminaries} we give some necessary definitions on monomial ideals and simplicial complexes. Section \ref{sec:ideals} describes the operations of polarization and depolarization on monomial ideals.  On Section \ref{sec:complexes} we describe the operations of polarization and depolarization of simplicial complexes. Section \ref{sec:AlexanderDual} presents our algorithms for Alexander dual computation and demonstrates its efficiency on several computer experiments. Finally, in Section \ref{sec:conclusions} we present some conclusions and lines for further work.

 \section{Preliminaries on monomial ideals and simplicial complexes}\label{sec:preliminaries}
 \subsection{Basic definitions and properties of monomial ideals}\ \\ 
 Let $R=\kb[x_1,\dots,x_n]$ be a polynomial ring in $n$ variables over a field $\kb$. A monomial ideal $I\subseteq R$ is an ideal of $R$ that has a generating set formed by monomials. Any monomial ideal $I$ has a unique minimal monomial generating set, denoted by $G(I)$. The importance of monomial ideals in commutative algebra is twofold. On the one hand, they are a practical tool to solve problems in conmutative algebra and algebraic geometry by reducing problems on general polynomial ideals and modules to monomial ones via Gr\"obner basis theory \cite{E95, HH11,V15}. On the other hand, their combinatorial nature allows the use of algebraic techniques to solve problems in other areas of combinatorics, such as graph theory or combinatorial algebraic topology, to name only two examples; see e.g. \cite{MS05,V13,FMS14,V15}. A remarkable result in this vein is the proof of the Upper Bound Conjecture for sphere triangulations by R. Stanley \cite{S75}. 

 The following notations and definitions will be used in this paper.
 \begin{itemize}
     \item[-] A monomial $\xb^\mu=x_1^{\mu_1}\cdots x_n^{\mu_n}$ is {\em squarefree} if $\mu_i\leq1$ for all $i\in\{1,\dots,n\}=[n]$. A monomial ideal $I$ is squarefree if its minimal monomial generating set consists of squarefree monomials. As is usual in the literature, we identify a monomial $\xb^\mu$ with its exponent multi-index $\mu$.
     \item[-] The {\em support} of a monomial $\xb^\mu$ is the set of indices of the variables that divide $\xb^\mu$, i.e., $\supp(\xb^\mu)=\left\{i\in[n]\mid \mu_i>0\right\}$. We also mean by support of $\xb^\mu$ the set of variables whose indices are in $\supp(\xb^\mu)$. The support of a monomial ideal $I\subset R$ is $\supp(I)=\cup_{m\in G(I)}\supp(m)$. We say $I$ has \emph{full support} if $\supp(I)=[n]$. We will assume that ideals have full support, unless otherwise stated.
     \item[-] For any set $\sigma\subseteq [n]$, the squarefree monomial $\prod_{i\in\sigma}x_i$ is denoted by $\xb_\sigma$.
     \item[-] $\xb^{\mu_I}=\lcm(\xb^\mu\mid \xb^\mu\in G(I))$ and $\xb^{\nu_I}=\gcd(\xb^\mu\mid \xb^\mu\in G(I))$ are, respectively, the monomials given by the least common multiple and greatest common divisor of the minimal monomial generators of $I$.
 \end{itemize}

Every monomial ideal $I$ can be seen as a (multi) graded $R$-module, and as such, we can define its Betti numbers (see \cite{E95,HH11,MS05}) as 
\begin{equation*}
    \beta_{i}(I)=\dim_{\kb}H_i(\Tor(I,\kb)).
\end{equation*}
Since $I$ is an $\NN^n$-multigraded module, then $\Tor(I,\kb)$ is $\NN^n$-multigraded too, and so are its homology modules, hence for each $\mu\in\NN^n$ we define $\beta_{i,\mu}(I)$ as the dimension of the multi-degree $\mu$ piece of $H_i(\Tor(I,\kb))$.

\subsection{Stanley-Reisner correspondence and Koszul complexes}\ \\ 
Squarefree monomial ideals are in one-to-one correspondence to simplicial complexes. This was defined by Stanley and Reisner \cite{S96} and was used to prove the Upper Bound Conjecture for sphere triangulations \cite{S75}; this correspondence has given rise to a fruitful area of research, see the class \texttt{13F55} in the 2020 Mathematics Subject Classification (see \url{https://msc2020.org/}).
\begin{definition}
    Let $\Delta$ be an abstract simplicial complex in the set of vertices $[n]$. The {\em Stanley-Reisner ideal} of $\Delta$ is defined as
    \[
    I_\Delta=\langle\xb_\sigma\mid\sigma \mbox{ is a minimal non-face of } \Delta\rangle.
    \]
    Conversely, given a squarefree monomial ideal $I$, the simplicial complex $\Delta_I$ such that $I=I_{\Delta_I}$ is called the Stanley-Reisner complex of $I$.
\end{definition}
The homology of $\Delta$ and $I_\Delta$ are related by Hochster's formula \cite{H77} and its variants. In our context, we will use the following formulation:
\begin{theorem}
Let $\Delta$ be an abstract simplicial complex, and $I_\Delta$ its Stanley-Reisner ideal, then
\[
\dim_\kb{{\widetilde{H}}_{i}(\Delta;\kb)}=\beta_{n-i-2}(I_\Delta).
\]
\end{theorem}

The correspondence between simplicial complexes and monomial ideals can be extended from squarefree to general monomial ideals by means of the Koszul simplicial complex \cite{MS05}, also related via the Alexander dual.
\begin{definition}
    Let $I\subseteq R=\kb[x_1,\dots,x_n]$ be a monomial ideal and $\mu\in\NN^n$ a multi-index. The (upper) Koszul simplicial complex of $I$ at $\mu$ is defined as
    \[
    K^\mu_I=\left\{\sigma\subseteq[n]\mid \frac{\xb^\mu}{\xb_\sigma}\in I\right\}.
    \]
\end{definition}

\begin{definition}
Let $\Delta$ be an abstract simplicial complex with ground set $V$. The Alexander dual of $\Delta$ is the abstract simplicial complex defined by: \[\Delta^{\vee} = (V,\{\sigma \subseteq V \; | \; V \setminus \sigma \not \in \Delta\}).\]
\end{definition}

In the squarefree case, we recover the Stanley-Reisner correspondence, since $K^\1_I=\Delta^\vee_I$, where $\Delta^\vee$ is the Alexander dual of $\Delta$.

Using the combinatorial Alexander duality \cite{BT09} for $\Delta_I$ we have the following relation.

\begin{theorem}
\label{th:Hochster}
Let $I\subseteq R=\kb[x_1,\dots,x_n]$ be a squarefree monomial ideal, and $\Delta_I$ its Stanley-Reisner complex. Then
    \[
    \beta_{i,\1}(I)=\dim_\kb\widetilde{H}_{i-1}(K^\1_I;\kb)=\dim_\kb\widetilde{H}_{i-1}(\Delta^\vee_I;\kb)=\dim_\kb\widetilde{H}_{n-i-2}(\Delta_I;\kb).
    \]
    The first equality can be extended to the non-squarefree case as \[ \beta_{i,\mu}(I)=\dim_\kb\widetilde{H}_{i-1}(K^\mu_I;\kb).\]
\end{theorem}

The notion of Alexander duality for abstract simplicial complexes can be defined for monomial ideals, see \cite{ER98,MS05,BT09b}. For a monomial $\xb^\mu\in R=\kb[x_1,\dots,x_n]$, denote by $\mk^\mu$ the ideal $\mk^\mu=\langle x_1^{\mu_1}, \dots, x_n^{\mu_n}\rangle$ (where only those indices $i$ such that $\mu_i>0$ are considered). 

\begin{definition}
    The {\em Alexander dual ideal} $I^\vee$ of the squarefree ideal $I$ is the (also squarefree) monomial ideal
    \[
        I^\vee=\bigcap_{\mu\in G(I)}\mk^\mu.
    \]

    In the case of non-squarefree monomial ideals, this notion can be extended in the following way, see \cite{MS05}. For $\mu,\nu\in\NN^n$ such that $\mu\geq\nu$ denote by $\mu\setminus\nu$ the multi-index given by
    \[
        (\mu\setminus\nu)_i=\mu_i+1-\nu_i,\mbox{ if } \nu_i>0,\mbox{ and } 0\mbox{ otherwise.}
    \]
    Then, the {\em Alexander dual ideal of the monomial ideal $I$ with respect to the monomial $\xb^\ab\geq \xb^{\mu_I}$} is 
    \[
        I^\vee_\ab=\bigcap_{\mu\in G(I)}\mk^{\ab\setminus\mu}.
    \]
\end{definition}

If $I$ is squarefree and $\ab=\1$, both definitions agree. If not otherwise stated, given a monomial ideal $I$ and $\mu_I$ the least common multiple of its minimal generators, Alexander duals will be considered with respect to $\mu_I$.

\section{Polarization and depolarization of monomial ideals}\label{sec:ideals}
\subsection{Polarization}
 \begin{definition}\label{def:polarization}
Let $\ab=(a_1,\dots,a_n)$ and $\mu=(\mu_1,\dots,\mu_n)$ be two multi-indices in $\NN^n$, where $\mu_i\leq a_i$ for all $i$. The polarization of $\mu$ in $\NN^{a_1+\cdots+a_n}$ is the multi-index $$\Pc(\mu)=(\underbrace{1,\dots,1}_{\mu_1},\underbrace{0,\dots,0}_{a_1-\mu_1},\dots,\underbrace{1,\dots,1}_{\mu_n},\underbrace{0,\dots,0}_{a_n-\mu_n}).$$ The {\em polarization of} $\xb^\mu=x_1^{\mu_1}\cdots x_n^{\mu_n}\in R=\kb[x_1\dots,x_n],$ with respect to $\ab$ is the squarefree monomial $\Pc(\xb^\mu)=\xb^{\Pc(\mu)}=x_{1,1}\cdots x_{1,\mu_1}\cdots x_{n,1}\cdots x_{n,\mu_n}$ in the polynomial ring $\Pc(R)=\kb[x_{1,1},\dots,x_{1,a_1},\dots,x_{n,1},\dots,x_{n,a_n}]$ in $n'=\sum_{i=1}^na_i$ variables. For ease of notation we used $\xb$ with two different meanings.

Let $I=\langle \xb^{m_1},\dots,\xb^{m_r}\rangle\subseteq R$ be a monomial ideal and $\xb^{\mu_I}=\lcm(\xb^{m_1},\dots,\xb^{m_r})$.
The {\em polarization of $I$}, $\Pc(I)$, is the monomial ideal in $\Pc(R)$ given by $\Pc(I)=\langle \Pc(\xb^{m_1}),\dots,\Pc(\xb^{m_r})\rangle$, where $\Pc(\xb^{m_i})$ is the polarization of $\xb^{m_i}$ with respect to $\mu_I$. 
\end{definition}

The polarization operation has been widely studied, mainly because of the features that a monomial ideal $I$ and its polarization $\Pc(I)$ share, see e.g. \cite{F05,HH11,V15,AV22}. 
The following proposition gives a list of some important copolar properties (i.e. shared by a monomial ideal and its polarization), see also \cite{F05,BT09,MMVV19,AFL22}.

\begin{proposition}[Corollary 1.6.3 in \cite{HH11}]\label{prop:copolarProperties}
Let $I\subseteq R=\kb[x_1,\dots,x_n]$ be a monomial ideal and let $\Pc(I)\subseteq \Pc(R)$ be its polarization. Then
\begin{enumerate}
\item $\beta_{i,\mu}(I)=\beta_{i,\Pc(\mu)}(\Pc(I))$ for all $i\in\ZZ$ and $\mu\in\NN^n$
\item Let $H_I(t)$ denote the Hilbert function of $I$, then $H_{I}(t)=(1-t)^\delta H_{\Pc(I)}(t)$ where $\delta=\dim \Pc(R)-\dim R$
\item $\height(I)=\height(\Pc(I))$
\item $\projdim(I)=\projdim(\Pc(I))$ and $\reg(I)=\reg(\Pc(I))$, where $\projdim$ and $\reg$ are the projective dimension and Castelnuovo-Mumfrod regularity, respectively.
\item $I$ is Cohen-Macaulay (resp. Gorenstein) if and only if $\Pc(I)$ is Cohen-Macaulay (resp. Gorenstein).
\end{enumerate}
\end{proposition}

\subsection{Depolarization}\label{sec:depolarization}
The polarization of a monomial ideal is unique, but the converse operation, {\em depolarization}, is not. Polarization takes a general monomial ideal $I$ and produces a squarefree monomial ideal, $\Pc(I)$, so that we can use topological and combinatorial techniques to study $I$, which are often more accessible. Conversely, depolarization takes a squarefree monomial ideal $J$ and gives a family of monomial ideals with the main homological features of $J$, but in rings with smaller number of variables. When applied to simplicial complexes, this allows to study them via complexes in less vertices. 
All notations and definitions are taken from \cite{MPSW20,PS22}.

\begin{definition}
Let $R,S$ and $T$ be polynomial rings over a field $\kb$. Let $I\subseteq R$ be a squarefree monomial ideal. A {\em depolarization of $I$} is a monomial ideal $J\subseteq S$ such that $I$ is isomorphic to $\Pc(J)\subseteq T\cong R$, i.e., there is a bijective map $\varphi$ from the variables of $R$ to the variables of $T$ such that $\varphi(G(I))=G(\Pc(J))$.
\end{definition}



In order to list the elements of the set of depolarizations of a given squarefree monomial ideal and describe a structure for this set, we introduce the following concepts.

\begin{definition}
    Let $I \subseteq R = \kb[x_{1},\dots,x_{n}]$ be a squarefree monomial ideal with minimal generating set $G(I)$. For each $i\in\supp(I)$ we define  $C_i\subseteq\supp(I)$ as \[C_i=\bigcap_{\xb^{\mu}\in G(I)\atop x_i\text{ divides }\xb^{\mu}}\supp(\xb^{\mu}).\]
    In other words, $C_i$ is given by the indices of all the variables that appear in every monomial generator of $I$ in which $x_i$ is present. Let $C_I=\{C_1,\dots,C_n\}$. The poset on the elements of $C_I$ ordered by inclusion is called the \emph{support poset of $I$}, $\supp\text{Pos}(I)$. We define the \emph{support poset of $I$}, where $I$ is a general monomial ideal, as the support poset of its polarization.
The support poset of any monomial ideal $I$, together with a given ordering $<$ on the variables $x_1,\dots,x_n$ induces a \emph{partial order $\prec$} on the set of variables as follows: $x_i\prec x_j$ if $C_i\subset C_j$ or if $C_i=C_j$ and $x_i<x_j$. This defines the \emph{$<$-support poset of $I$}. 
\end{definition}
From the $<$-support poset of $I$ we obtain the set of depolarizations of $\Pc(I)$: each disjoint chain partition of the $<$-support poset gives a depolarization of $\Pc(I)$, see \cite{MPSW20}, Proposition 3.5 and Theorem 3.7.
The set of depolarizations of a squarefree monomial ideal forms a poset with the order given by refinement of the corresponding chain partition.

\begin{definition}
Let $I\subset R=\kb[x_1,\dots,x_n]$ be a squarefree monomial ideal and let $J$ and $J'$ be two depolarizations of $I$. We say that $J\leq J'$ if the chain partition giving rise to $J$ is a refinement of the one corresponding to $J'$. Using this ordering, a collection of ideals that are depolarizations of a given squarefree monomial ideal $I$ forms a poset in which $I$ is the unique minimal element. We call this the \emph{depolarization poset of $I$}, $\mathcal{D}\text{P}(I)$.
The depolarization poset of a general monomial ideal $J$ is defined as the depolarization poset of its polarization $\Pc(J)$, i.e., $\mathcal{D}\text{P}(J):=\mathcal{D}\text{P}(\Pc(J))$.
\end{definition}

\begin{example}\label{Example:DePolarization}
    Let $J=\langle x^3y,yz^3,x^2y^3z^2t,z^3t\rangle\subset\kb[x,y,z,t]$ be a monomial ideal. Its polarization is 
    \begin{multline*}
        \mathcal{P}(J)=\langle x_1x_2x_3y_1,y_1z_1z_2z_3,x_1x_2y_1y_2y_3z_1z_2t_1,z_1z_2z_3t_1\rangle\\
        \subset\kb[x_1,x_2,x_3,y_1,y_2,y_3,z_1,z_2,z_3,t_1].
    \end{multline*}
    Observe that the polarization of $J$, $\mathcal{P}(J)$, is a monomial ideal in a polynomial ring of $10$ variables, while $J$ lays in a polynomial ring of just $4$ variables. 
     For ease notation we will work with 
    \[
        I=\langle x_1x_2x_3x_4, x_4x_7x_8x_9, x_1x_2x_4x_5x_6x_7x_8x_{10}, x_7x_8x_9x_{10}\rangle\subset\kb[x_1,\dots,x_{10}],
    \] 
    which is isomoprphic to $\Pc(J)$.

    The support poset $(C_I,\subset)$ is formed by the following sets 
    \begin{align*}
        C_1&=C_2=\{1,2,4\},C_3=\{1,2,3,4\}, C_4=\{4\},\\
        C_5&=C_6=\{1,2,4,5,6,7,8,10\},\\
        C_7&=C_8=\{7,8\}, C_9=\{7,8,9\}\text{ and }C_{10}=\{7,8,10\}.
    \end{align*}

    The Hasse diagram of the support poset of $I$ is shown in Figure \ref{fig:DePolarizationExPoset}.  Using the order in the variables in which $x_2>x_1$, $x_6>x_5$ and $x_8>x_7$ we obtain the $<$-support poset of $I$, which is presented in Figure \ref{fig:DePolarizationExOrderedPoset}.
    
     From this $<$-support poset, we can obtain different depolarizations of $I$. For example, using the partition $ P_{I_1}=\{\{1,2,3,4\},\{5,6,7,8,10\},\{9\}\}$ we obtain the depolarization $I_1=\langle x^4,xy^2z,x^3y^5,y^3z\rangle\subset\kb[x,y,z]$ . Refining this partition to the one given by $P_{I_2}=\{\{1,2,3,4\},\{5,6,10\},\{7,8\},\{9\}\}$ we obtain the depolarization $I_2=\langle x^4,xz^2t,x^3y^3z^2,yz^2t\rangle\subset\kb[x,y,z,t]$.
             
             



    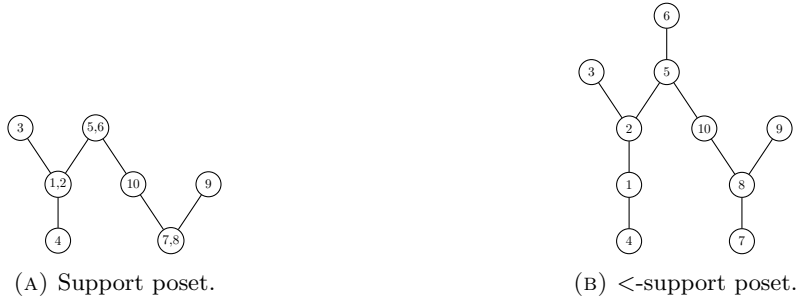
\begin{figure}
     \centering
         \begin{subfigure}[b]{0.40\textwidth}
             \centering
                 \begin{tikzpicture}[scale=0.5, transform shape, vertices/.style={text width= 1.5em,align=center,draw=black, fill=white, circle, inner sep=1pt}]
             
                     \node [vertices] (A) at (2,0){4};
                     \node [vertices] (B) at (5,0){7,8};
                     \node [vertices] (C) at (2,1.5){1,2};
                     \node [vertices] (D) at (1,3){3};
                     \node [vertices] (E) at (3,3){5,6};
                     \node [vertices] (F) at (4,1.5){10};
                     \node [vertices] (G) at (6,1.5){9};
                     
                    \draw [-] (A)--(C);
                    \draw [-] (C)--(D);
                    \draw [-] (C)--(E);
        
                    \draw [-] (B)--(F);
                    \draw [-] (B)--(G);
                    \draw [-] (F)--(E);
    
                \end{tikzpicture}
             \caption{Support poset.}
             \label{fig:DePolarizationExPoset}
         \end{subfigure}
        \hfill
         \begin{subfigure}[b]{0.40\textwidth}
             \centering
              \begin{tikzpicture}[scale=0.5, transform shape, vertices/.style={text width= 1.5em,align=center,draw=black, fill=white, circle, inner sep=1pt}]
             
                 \node [vertices] (A) at (2,0){4};
                 \node [vertices] (B) at (5,0){7};
                 \node [vertices] (B1) at (5,1.5) {8};
                 \node [vertices] (C) at (2,1.5){1};
                 \node [vertices] (C1) at (2,3){2};
                 \node [vertices] (D) at (1,4.5){3};
                 \node [vertices] (E) at (3,4.5){5};
                 \node [vertices] (E1) at (3,6){6};
                 \node [vertices] (F) at (4,3){10};
                 \node [vertices] (G) at (6,3){9};
                 
                \draw [-] (A)--(C);
                \draw [-] (C)--(C1);
                \draw [-] (C1)--(D);
                \draw [-] (C1)--(E);
                \draw [-] (E)--(E1);
    
                \draw [-] (B)--(B1);
                \draw [-] (B1)--(F);
                \draw [-] (B1)--(G);
                \draw [-] (F)--(E);

            \end{tikzpicture}
            \caption{$<$-support poset.}
            \label{fig:DePolarizationExOrderedPoset}
         \end{subfigure}
    \caption{Support poset and $<$-support poset of the monomial ideal $I$ in Example \ref{Example:DePolarization}.}
    \label{fig:exampleDepolarization}
\end{figure}

             
             


\end{example}

\section{Polarization and depolarization of simplicial complexes}\label{sec:complexes}

\subsection{Polarization of Koszul simplicial complexes}\label{sec:ploarizationKoszul}\ \\ 
Let $I\subseteq R=\kb[x_1,\dots,x_n]$ be a monomial ideal and $G(I)=\{\xb^{m_1},\dots,\xb^{m_r}\}$ its minimal monomial generating set. Let $\xb^{\mu_I}$ and $\xb^{\nu_I}$ be the least common multiple and greatest common divisor of all the monomials in $G(I)$. If $I$ is clear from the context we will denote then simply by $\xb^\mu$ and $\xb^\nu$. Observe that $\xb^\nu\mid \xb^\mu$ and that equality holds only if $I$ is generated by one single monomial. We set $\ms(I)=\frac{\xb^{\mu}}{\xb^\nu}$ and call it the {\em monomial span} (or simply {\em span}) of $I$.

Consider now $K^\mu_I$, the Koszul simplicial complex of $I$ at $\mu$ and construct the complex $EK^\mu_I$, called the {\em expanded Koszul complex of $I$ at $\mu$} using the following steps:

\begin{enumerate}
    \item[$1$.]{For each $i\in[n]$ consider the set of vertices $E_i=\{x_{i,1},\dots,x_{i,\ms(I)_i}\}$. Within each set $E_i$ we consider the vertices ordered increasingly according to the second subindex. The set of vertices of $EK^\mu_I$ is $EV=\bigcup_{i}E_i$.}
    \item[$2$.]{For each $\sigma\in K^\mu_I$, let $\xb^{\alpha(\sigma)}$ be each of the biggest monomial divisors of $\ms(I)$ (w.r.t divisibility) such that $\supp(\xb^{\alpha(\sigma)})=\sigma$} and $\frac{\xb^\mu}{\xb^{\alpha(\sigma)}}\in I$. We add to $EK^\mu_I$ a face given by the set $\bigcup_{i\in\sigma}\{x_{i,\ms(I)_i-\alpha(\sigma)_i+1},\dots,x_{i,\ms(I)_i}\}$, i.e., the simplex spanned by the last $\alpha(\sigma)_i$ vertices in $E_i$ for each $i\in\sigma$.
\end{enumerate}

\noindent Alternatively, the second step of this construction can be replaced by
\begin{enumerate}
    \item[$2^\prime$.] For each $m \in G(I)$, let $\xb^{\alpha} = \frac{\xb^{\mu}}{m}$ and $\sigma = \supp(\xb^{\alpha})$. We add to $EK^\mu_I$ a facet given by the set $\bigcup_{i\in\sigma}\{x_{i,\ms(I)_i-\alpha_i+1},\dots,x_{i,\ms(I)_i}\}$, i.e., the simplex spanned by the last $\alpha_i$ vertices in $E_i$ for each $i$.
\end{enumerate}

Using (2) to describe the facets of $EK^\mu_I$ stresses the parallelism with the definition of the Koszul complex, whereas using (2') relates directly to the monomial generators of $I$.

\begin{example}\label{ex:expansion}
     Let $I=\langle x_1^3x_2^2,x_1^2x_2^3,x_1^2x_3,x_2^2x_3\rangle\subseteq \kb[x_1,x_2,x_3]$. For this ideal $\xb^\mu=x_1^3x_2^3x_3$, $\xb^\nu=1_R$ and $\ms(I)=x_1^3x_2^3x_3$.
     Observe that the facets of $K^\mu_I$ are
     \[
     \Fc \left( K^\mu_I \right)=\{\{x_1,x_2\},\{x_1,x_3\},\{x_2,x_3\}\}.
     \]
     
     The vertex set for $EK^\mu_I$ is $\{x_{1,1},x_{1,2},x_{1,3},x_{2,1},x_{2,2},x_{2,3},x_{3,1}\}$, and the simplicial complex is given by the following set of facets (see Figure \ref{fig:exampleExpansion})
     \[
     EK^\mu_I\!=\!\{\{x_{1,1},x_{1,2},x_{1,3},x_{2,3}\},\{x_{1,3},x_{2,1},x_{2,2},x_{2,3}\},\{x_{1,3},x_{3,1}\},\{x_{2,3},x_{3,1}\}\}.
     \]

\begin{figure}[h]
\centering
\begin{minipage}{.40\textwidth}
  \centering
  \includegraphics[width=.5\linewidth]{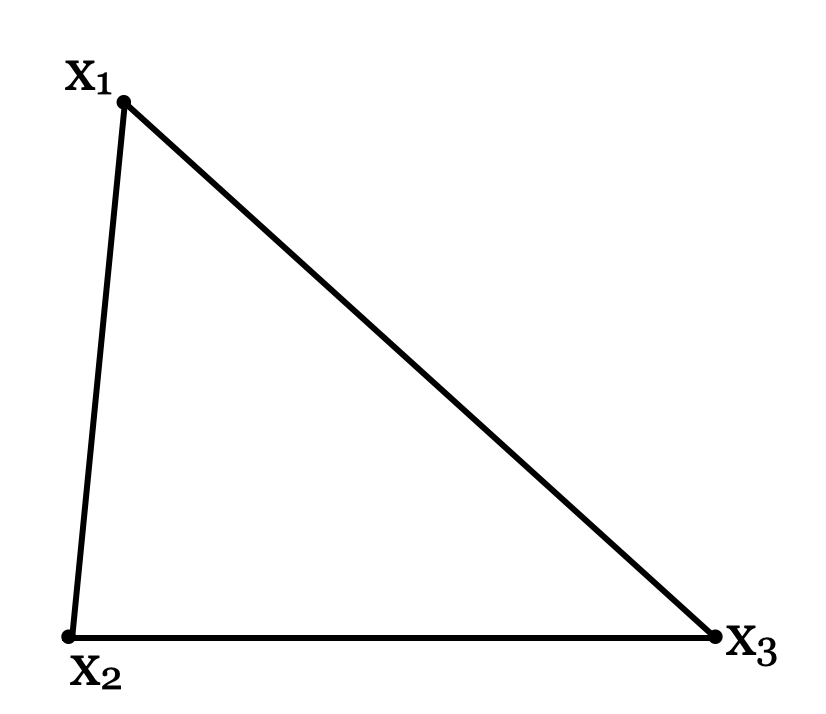}
  
\end{minipage}%
\begin{minipage}{.60\textwidth}
  \centering
  \includegraphics[width=0.7\linewidth]{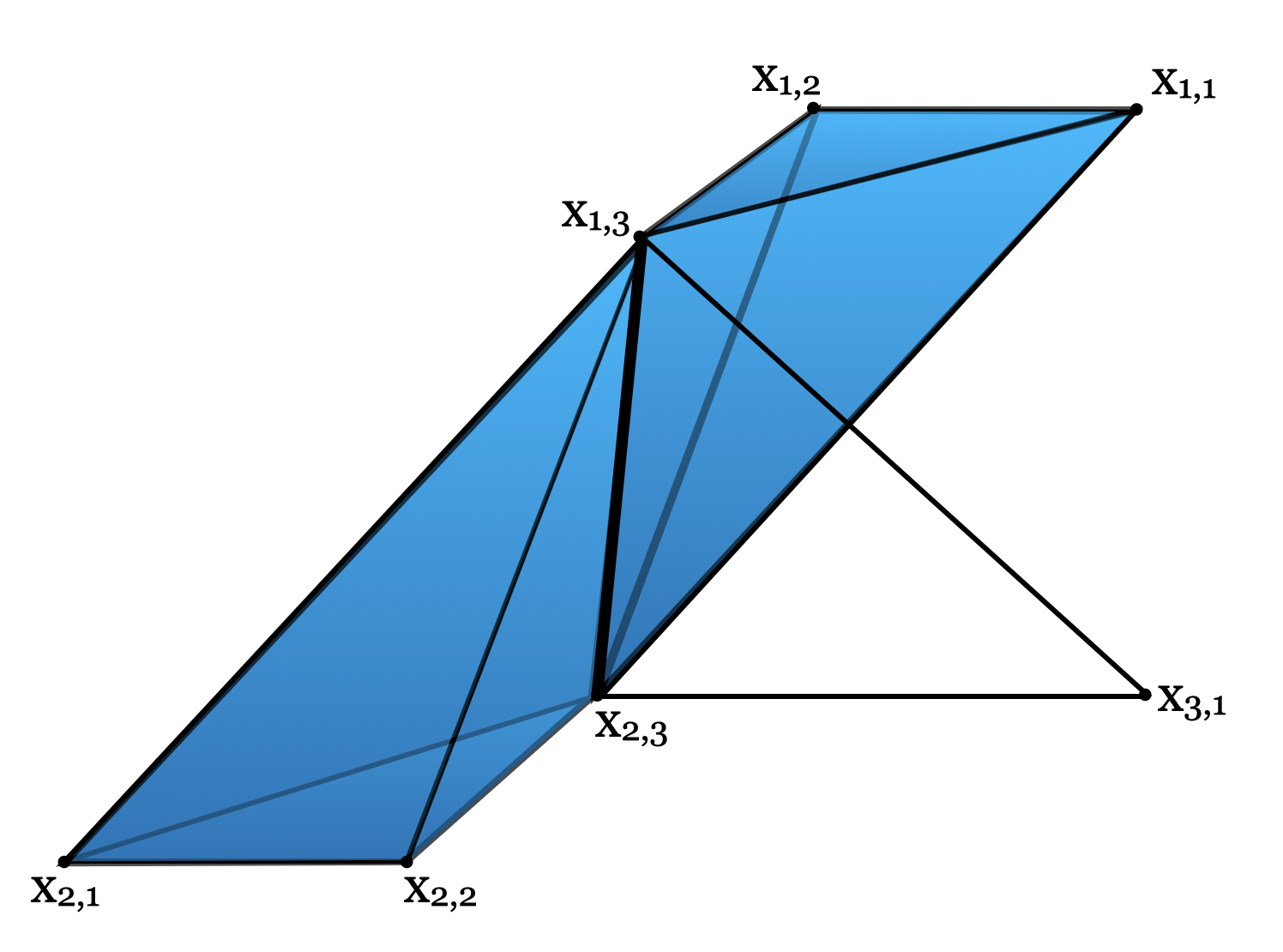}
  
\end{minipage}
\caption{A simplicial complex (left) and its expansion, from Example \ref{ex:expansion}.}
\label{fig:exampleExpansion}
\end{figure}
\end{example}

%

\begin{proposition}\label{prop:expandedBetti}
Let $I\subseteq R=\kb[x_1,\dots,x_n]$ be a monomial ideal. Then
\[
\dim_\kb\widetilde{H}_i(K^\mu_I;\kb)=\dim_\kb\widetilde{H}_i(EK^\mu_I;\kb),\mbox{ for all }i\in\ZZ.
\]
\end{proposition}
\begin{proof}
Let $\sigma \in K^{\mu}_{I}$. Then $\frac{\xb^{\mu}}{\xb_\sigma} \in I$ and $ \xb_\sigma \;| \; \xb^{\alpha(\sigma)}$, hence the face given by $\bigcup_{i\in\sigma}\{x_{i,\ms(I)_i}\}$ is in $EK^{\mu}_{I}$. The reciprocal is also true; that is, there is a one-to-one correspondence between the faces $\sigma \in K_{I}^{\mu}$ and those faces of $EK_{I}^{\mu}$ of the form $\bigcup_{i\in\sigma}\{x_{i,\ms(I)_i}\}$. This means that inside $EK_{I}^{\mu}$ there is a subcomplex isomorphic to $K_{I}^{\mu}$, let's denote it by $maxEK_{I}^{\mu}$, for its vertex set is formed by the maximal vertex in each of the sets $E_i$. Our goal is to show that there is a sequence of elementary collapses from $EK_{I}^{\mu}$ to $maxEK_{I}^{\mu}\simeq K_{I}^{\mu}$.


Let $i\in[n]$, for any face $\sigma$ in $EK_{I}^{\mu}$ containing at least one vertex in $E_i$ there is at least one face $\sigma'$ such that $\sigma\subseteq\sigma'$ and $x_{i,ms(I)_{i}}\in\sigma'$. Hence, we can write the union of all the faces containing at least one vertex of $E_{i}$ as $\bigcup_{j=1}^{ms(I)_{i}} \langle \sigma_{i,j}\rangle * A_{j}$, where $\sigma_{i,j} = \{x_{i,j},\dots,x_{i,ms(I)_{i}}\}$ and $V(A_{j}) \cap E_{i} = \emptyset, \; \forall j$ (note that $A_{j}$ can be $\emptyset$), and where $*$ denotes the join operator.

Consider $i=1$. Since any simplex can be collapsed to any of its vertices and $x_{1,ms(I)_1} \in \langle \sigma_{1,j}\rangle$ for all $j$, we can collapse these sets to $x_{1,ms(I)_1}$. This collapse is done stepwise following the second index of the vertices in ascending order until we reach the last one, $x_{1,ms(I)_1}$. Then by the compatibility of $*$ with respect to collapses (see e.g. \cite{S19} Problem 1.55), we have that $\bigcup_{j=1}^{ms(I)_{1}} \langle \sigma_{1,j}\rangle * A_{j} \searrow \langle\{x_{1,ms(I)_{1}}\}\rangle * \bigcup_{j=1}^{ms(I)_{1}}A_{j}$. The only surviving vertex of $E_1$ after this collapse is $x_{1,ms(I)_1}$, and every vertex of $E_j$, $j> 1$ remains. If we call $EK_{I}^{\mu}\setminus\{i\}$ to the result of applying the just described collapse of $E_i$ to $x_{i,ms(I)_i}$ then we have that $EK_{I}^{\mu}\searrow EK_{I}^{\mu}\setminus\{1\}$ (hence both have the same homotopy type), and that $maxEK_{I}^{\mu}\subseteq EK_{I}^{\mu}\setminus\{1\}$.



This course of action can be repeated for every $i=1,\dots,n$, obtaining a subcomplex $EK_{I}^{\mu}\setminus\{1,\dots,i\}$ such that $EK_{I}^{\mu}\searrow EK_{I}^{\mu}\setminus\{1,\dots,i\}$, and such that the only vertex in $E_j$ for $j=1,\dots,i$ is $x_{j,ms(I)_j}$, and such that $maxEK_{I}^{\mu}\subseteq EK_{I}^{\mu}\setminus\{1,\dots,i\}$. Finally, noting that $EK_{I}^{\mu}\setminus[n]=maxEK_{I}^{\mu}$ yields the result.

\end{proof}

\begin{proposition}\label{prop:expandedIsKoszul}
 Let $I\subseteq R$ be a monomial ideal and $\Pc(I)\subseteq\Pc(R)$ its polarization, then
 \[
 K^{\Pc(\mu_I)}_{\Pc(I)} \cong EK^{\mu_I}_I.
 \]
\end{proposition}
\begin{proof}
    Let $G(I)=\{\xb^{m_1},\dots,\xb^{m_r}\}$. Since $\Pc(\mu_I) = \1$, the facets of $K^{\Pc(\mu_I)}_{\Pc(I)}$ are $\{ \1 - \Pc(m_i) : i \in \{ 1, \ldots, r \}\}$. Taking into account steps $1$ and $2^\prime$ of the construction of $EK^{\mu_I}_I$, it is clear that $\varphi : V \left( EK^{\mu_I}_I \right) \rightarrow V \left( K^{\Pc(\mu_I)}_{\Pc(I)} \right)$, defined by $\varphi(x_{i,j}) = x_{i,j+\nu_i}$, induces an isomorphism between $EK^{\mu_I}_I$ and $K^{\Pc(\mu_I)}_{\Pc(I)}$.
\end{proof}

Propositions \ref{prop:expandedBetti} and \ref{prop:expandedIsKoszul}, together with Hochster's formula give a proof of item (1), and hence item (4), in Proposition \ref{prop:copolarProperties}, and also give an explicit construction of the relation between $K^{\mu_I}_I$ and $K^{\Pc(\mu_I)}_{\Pc(I)}$ as simplicial complexes.

\subsection{Depolarization of simplicial complexes}\ \\ 
The study of simplicial complexes, in particular with a focus on efficient computation of the dimension of their homology groups and related invariants such as Euler characteristic, has been approached, from the theory of monomial ideals \cite{HH11,V15,BHS19,RS13}, and from several reduction techniques with the aim of obtaining smaller simplicial complexes having the same homology than the one under study; we underline the notions of elementary collapse, introduced by Whitehead \cite{W39} (see also \cite{BG12}) and Discrete Morse Theory (both in a topological and algebraic formulation) \cite{F02,S19,BV02,S06}. The use of depolarization is a mix of these two approaches: it uses monomial ideals to reduce a given simplicial complex so that its homology can be more efficiently computed.

Let $\Delta$ be a simplicial complex on $n$ vertices. Let $IK_\Delta$ be the monomial ideal minimally generated by the complements of the facets of $\Delta$ in $[n]$, i.e.,
\[
IK_\Delta=\Big{\langle} \frac{\xb^\1}{\xb_\sigma}\mid\sigma\in\Delta\Big{\rangle}.
\]
Observe that $IK_\Delta=I_\Delta^\vee=I_{\Delta^\vee}$. 

 \begin{proposition}\label{prop:simplicialDepolarization}
    Let $\Dc(IK_\Delta)$ be a depolarization of $IK_\Delta$.
\[
\dim_\kb\widetilde{H}_i(\Delta;\kb)=\dim_\kb\widetilde{H}_i(K^\mu_{\Dc(IK_\Delta)};\kb)\mbox{ for all } i \in\ZZ.
\]
\end{proposition}
\begin{proof}
For every $i\in\ZZ$ we have that
\[
\dim_\kb{{\widetilde{H}}_{i}(\Delta;\kb)} = \beta_{i+1,\1} (IK_\Delta) = \beta_{i+1,\mu} (\Dc(IK_\Delta)) = \dim_\kb{{\widetilde{H}}_{i}(K^\mu_{\Dc(IK_\Delta)};\kb)}.
\]

The first equality holds because $K^\1_{IK_\Delta} = \Delta$ and due to Hochster's formula (Theorem \ref{th:Hochster}). The seconds holds because depolarization keeps the Betti numbers of ideals. The third one is again due to Hochster's formula.
\end{proof}

Observe that if  the depolarization is not trivial, then $K^\mu_{\Dc(IK_\Delta)}$ is a simplicial complex in a strictly smaller set of vertices than $\Delta$.

\begin{example}
Let $\Delta$ be the following $4$-dimensional simplicial complex in $6$ vertices (given by its facets)
\[
 \Delta=\{\{1,2,3,4,5\},\{1,2,3,6\},\{4,5,6\}\},
 \]
whose f-vector is $(1,6,15,14,6,1)$.
The monomial ideal $IK_\Delta\subseteq\kb[x_1,\dots,x_6]$ is minimally generated by the set 
\[
G(IK_\Delta)=\{x_6,x_4x_5,x_1x_2x_3\}.
\]

The ideal $J=\langle a,b^2,c^3\rangle\subseteq\kb[a,b,c]$ is a depolarization of $IK_\Delta$. The facets of $K^\1_J$ are $\{\{a,b\},\{a,c\},\{b,c\}\}$, $K^\1_J$ is an empty triangle in three vertices, hence its f-vector is $(1,3,3)$, and $\dim_\kb\widetilde{H}_i(K^\1_J;\kb)=1$ if $i\in\{0,1\}$, and it is $0$ otherwise. This is also the homology of the original complex $\Delta$.
\end{example}
\begin{remark}
Depolarization of simplicial complexes can be seen as finding groups of vertices that form  a simplex, such that they can be sorted so that the whole simplex collapses to the last vertex. The described algebraic approach is a way to find such groups. Observe that not every collapse in a simplicial complex corresponds to a depolarization of its Stanley-Reisner ideal.
\end{remark}

\section{Computing Alexander duality via depolarization}\label{sec:AlexanderDual}
 Computing the facets of the Alexander dual of a simplicial complex or a monomial ideal 
is known to be equivalent to the hypergraph transversal problem \cite{EGM02}. We can represent the simplicial complex $\Delta$ by its set of facets $\Fc(\Delta)$. The set of minimal non-faces of $\Delta$ corresponds to the minimal generators of the ideal $I_\Delta$, and also to the set of minimal transversals (or {\em minimal hitting sets}) of the family of complements of the facets of $\Delta$. Let $\Hc=\{V\setminus F \vert F \in \Fc(\Delta)\}$, the facets of $\Delta^\vee$ are the minimal transversals of the set family $\Hc$.
The hypergraph transversal problem is, in turn, equivalent to the monotone boolean dualization problem, whose complexity is quasi-polynomial on the sum of the sizes of the input and output \cite{RR25}. For these problems, the size of the output can be exponential on the size of the input. To see this, consider for instance the ideal $I_n=\langle x_iy_i\mid i\in[n]\rangle\subseteq\kb[x_1,\dots,x_n,y_1,\dots,y_n]$; the number of minimal generators of $I_n$ is $n$, and the number of generators of $I^\vee$ is $2^n$.  There exist some polynomial-time algorithms for many special classes of graphs and boolean functions, but still the exact complexity of the general problem is an open question.

 Given an abstract simplicial complex $\Delta$, monomial depolarization can be used to construct another complex $\Delta'$, smaller than $\Delta$, with the same Betti numbers. In general, this procedure can be used to efficiently compute features of $\Delta$ by means of $\Delta'$, which are preserved (or easily transcribed) under polarization and depolarization. We illustrate this procedure by the computation of the Alexander dual of a given simplicial complex. We take advantage of depolarization as a reduction of $\Delta$, in order to lower the input and output size of the problem. We propose Algorithm \ref{alg:AlexanderDual}, that performs this computation at the level of the corresponding Koszul ideals.

\begin{algorithm}[ht]
\caption{Alexander dual}
\label{alg:AlexanderDual}
	\KwData{A simplicial complex $\Delta$ given by its set of facets $\Fc(\Delta)$}
	\KwResult{The list of facets $\Fc(\Delta^\vee)$ of the Alexander dual $\Delta^\vee$}
	\Begin{
	$G(I_\Delta^\vee)\longleftarrow \{\xb_{[n]\setminus \sigma}\mid \sigma \in \Fc(\Delta)\}$\;
        $J\longleftarrow \mbox{ depolarization of }I_\Delta^\vee$\;
        Compute $J^\vee$\;
        From $J^\vee$ obtain $(I_\Delta^\vee)^\vee=I_\Delta$\;
        $\Fc(\Delta^\vee)\longleftarrow\{[n]\setminus\mu\mid \xb^\mu\in G({I_\Delta})\}$\;	
	\Return{$\Fc(\Delta^\vee)$
    }\;
        }
\end{algorithm}

Steps $2$ and $6$ of this algorithm have polynomial (in fact lineal) complexity on their respective inputs. Observe that in Step $2$ we obtain $IK_\Delta$, which equals $I_\Delta^\vee$. Step $3$ can be done in polynomial time with respect to the input, and so can be done step $5$. In step $4$ we face the exponential increase in the size of the ideal. However, since $J$ has a smaller number of variables and generators, the actual performance of this algorithm is expected to be better than to compute $\Delta^\vee$ directly from $I_\Delta^\vee$ (i.e., from $\Fc(\Delta)$).
Steps $2,3$ and $6$ of Algorithm \ref{alg:AlexanderDual} are straightforward and have been already described in the previous sections. For the computation of Alexander dual in step $4$, there are several available algorithms, being the one described in \cite{R09} the most efficient, to the best of our knowledge; this algorithm is the one implemented in state-of-the-art computer algebra systems like \texttt{CoCoA} \cite{CoCoA} and \texttt{Macaulay2} \cite{M2}. There remains to describe step $5$, i.e., to give a polynomial time algorithm to obtain the Alexander dual of a monomial ideal $I$ given the Alexander dual of a depolarization $J$ of $I$. This is done using the next proposition.

\begin{proposition}\label{prop:AlexanderPolarization}
Let $I$ be a monomial ideal. For each $\xb^\nu\in G(I^\vee)$ consider
\[
M^\nu=\{\prod_{i\in\supp(\nu)} x_{i,j_i}\mid j_i\in \{1,\dots,(\mu_I\setminus\nu)_i\} \ \forall i\}.
\]
Then, the Alexander dual ideal of $\Pc(I)$ is generated by
\[
\Pc(I)^\vee=\langle \bigcup_{\xb^\nu\in G(I^\vee)}M^\nu \rangle.
\]
\end{proposition}
\begin{proof}
Let $I=\langle \xb^{m_1},\dots,\xb^{m_r}\rangle\subseteq R$, and let $\mu=\mu_I$. Since $\1\setminus\Pc(m_j)=\Pc(m_j)$ for all $j$, we know that
\[
 I^\vee=\bigcap_{j=1}^r\mathfrak{m}^{\mu\setminus m_j}, \mbox{ and }    
 \Pc(I)^\vee=\bigcap_{j=1}^r\mathfrak{m'}^{\1 \setminus \Pc(m_j)}=\bigcap_{j=1}^r\mathfrak{m'}^{\Pc(m_j)},   
\]
where $\mathfrak{m}'$ is the ideal generated by the variables in $\Pc(R)$. Observe that $I^\vee$ and $\Pc(I)^\vee$ are, respectively, generated by
\begin{align*}
    G_1 &= \left\{ \lcm\{ a_j : j \in [r] \} \mid a_j \in G(\mathfrak{m}^{\mu\setminus m_j}) \ \forall j\right\}, \quad \text{and} \\
    G_2 &= \left\{ \lcm\{ a'_j : j \in [r] \} \mid a'_j \in G(\mathfrak{m'}^{\Pc(m_j)}) \ \forall j\right\}.
\end{align*}
Noting that $\mu\setminus(\mu\setminus m_j) = m_j$, that $G(\mathfrak{m}^{\mu\setminus m_j}) = \bigcup_{i\in\supp(m_j)}\{ x_i^{\mu_i+1-(m_j)_i}\}$, and that $G(\mathfrak{m'}^{\Pc(m_j)}) = \bigcup_{i\in\supp(m_j)}\{ x_{i,1}, \ldots, x_{i,(m_j)_i}\}$, we have that for every $x^\alpha \in G_1$, every monomial $\prod_{i\in\supp(\alpha)} x_{i,j_i}$, with $j_i \in [(\mu\setminus\alpha)_i]$ is in $G_2$ for each $i$. Since every element in $G_2$ is of this form or is redundant, and since redundant elements of $G_1$ produce redundant elements of $G_2$, the result holds.

\end{proof}
\begin{example}
Let $\Delta$ a simplicial complex in $10$ vertices whose set of facets and Koszul ideal are given by
\begin{align*}
\Fc(\Delta)=&\{\{v_5,v_6,v_7,v_8,v_9,v_{10}\},\{v_1,v_2,v_3,v_5,v_6,v_{10}\},\{v_3,v_9\},\\
&\{v_1,v_2,v_3,v_4,v_5,v_6\}\},\\
I_\Delta^\vee=&\langle v_1v_2v_3v_4,v_4v_7v_8v_9,v_1v_2v_4v_5v_6v_7v_8v_{10},v_7v_8v_9v_{10} \rangle\subseteq \kb[v_1,\dots,v_{10}].
\end{align*}

Let
$J=\langle x^4, x z^3, x^3 y^3 z^2, y z^3 \rangle\in R=\kb[x,y,z]$
a depolarization of $I$. Its Alexander dual is
$J^\vee=\langle x^4 y^3, x^2 z, xyz, x z^2 \rangle$.
The sets $M^\nu$ for $\xb^\nu\in G(J^\vee)$ are
\begin{align*}
    M^{x^4 y^3}&=\{x_1 y_1\},\\
    M^{x^2 z}&=\{x_i z_k \mid i \in \{1,2,3\}, k \in \{1,2,3\}\},\\
    M^{xyz}&=\{x_i y_j z_k \mid i \in \{1,2,3,4\}, j \in \{1,2,3\}, k \in \{1,2,3\}\},\\
    M^{x z^2}&=\{x_i z_k \mid i \in \{1,2,3,4\}, k \in \{1,2\}\}.
\end{align*}

Hence, the minimal generating set of the Alexander dual of $\Pc(J)$ is
\begin{align*}
 \Pc(J)^\vee=\langle& x_1y_1,x_1z_1,x_2z_1,x_3z_1,x_4z_1,x_1z_2,x_2z_2,x_3z_2,x_4z_2,x_1z_3,x_2z_3,x_3z_3,\\
 &x_4y_1z_3,x_4y_2z_3,x_4y_3z_3\rangle,
\end{align*}
which gives $I_\Delta$ under the correspondence
$x_1\mapsto v_4, x_2\mapsto v_1, x_3\mapsto v_2, x_4\mapsto v_3, y_1\mapsto v_{10}, y_2\mapsto v_5, y_3\mapsto v_6, z_1\mapsto v_7, z_2\mapsto v_8, z_3\mapsto v_9$.
Hence, we have
\begin{align*}
   \Fc(\Delta^\vee)=&\{\{v_1,v_2,v_3,v_5,v_6,v_7,v_8,v_9\},\{v_1,v_2,v_3,v_5,v_6,v_8,v_9,v_{10}\},\\
   &\{v_2,v_3,v_4,v_5,v_6,v_8,v_9,v_{10}\},\{v_1,v_3,v_4,v_5,v_6,v_8,v_9,v_{10}\},\\
   &\{v_1,v_2,v_4,v_5,v_6,v_8,v_9,v_{10}\},\{v_1,v_2,v_3,v_5,v_6,v_7,v_9,v_{10}\},\\
   &\{v_2,v_3,v_4,v_5,v_6,v_7,v_9,v_{10}\},\{v_1,v_3,v_4,v_5,v_6,v_7,v_9,v_{10}\},\\
   &\{v_1,v_2,v_4,v_5,v_6,v_7,v_9,v_{10}\},\{v_1,v_2,v_3,v_5,v_6,v_7,v_8,v_{10}\},\\
   &\{v_2,v_3,v_4,v_5,v_6,v_7,v_8,v_{10}\},\{v_1,v_3,v_4,v_5,v_6,v_7,v_8,v_{10}\},\\
   &\{v_1,v_2,v_4,v_5,v_6,v_7,v_8\},\{v_1,v_2,v_4,v_6,v_7,v_8,v_{10}\},\\
   &\{v_1,v_2,v_4,v_5,v_7,v_8,v_{10}\}\}.
\end{align*}
$J$ and $\Pc(J)$ have $4$ minimal generators. $J$ is an ideal in $4$ variables, and $\Pc(J)$ in $10$ variables. On the other hand, $J^\vee$ has $6$ generators in $4$ variables, while $\Pc(J)^\vee$ has $15$ generators in $10$ variables. Even in a small example like this, the advantage of computing the dual using the depolarization is apparent.

\end{example}
Note that the description of the sets $M^\nu$ in Proposition \ref{prop:AlexanderPolarization} may be very redundant, but their computation can be optimized proceeding by ascending levels on the poset of the supports of the monomials in $G(I^\vee)$ ordered by divisibility (or inclusion).

\subsection{Computer experiments}\ \\ 
In this section we present a series of computer experiments in order to show the difference in computing time for the dual of a monomial ideal versus the dual of its polarization. The goal of this section is to show that computing the Alexander dual is computationally more expensive for the polarization of an ideal than for the ideal itself. For this, we use several examples of different kinds, exploring the role of degree, regularity and the shape of the Betti diagram (see Appendix \ref{ap:Betti}). We raise the degrees of the monomials involved, so that the difference in resources demand becomes more evident therefore demonstrating the advantages of using a depolarization $J$ of the squarefree ideal $I_\Delta^\vee$ in step \texttt{4} of Algorithm \ref{alg:AlexanderDual}.

In all our examples, we generate several classes of monomial ideals and polarize them. This polarized ideal is thus a squarefree one, that plays the role of $I_\Delta^\vee$ in Algorithm \ref{alg:AlexanderDual}, while the original ideal plays the role of its depolarization $J$. For each ideal we make two comparisons: first we compare the {\em size}, in terms of number of generators, of the dual of the original ideal and the dual of its polarization; second, we compare the {\em time} taken to compute the dual of both ideals. In all the tables in this section, $\#(J)$ denotes the number of minimal generators of $J$ (analogously for $J^\vee$ and $(I_\Delta^\vee)^\vee=I_\Delta$). Recall that the number of generators of $I_\Delta^\vee$ is the same that the number of generators of $J$. Columns $t(J^\vee)$ and $t(I_\Delta)$ indicate the computation time in seconds for the computation of the Alexander dual of $J$ and $I_\Delta^\vee$ respectively. The times for computing the Betti numbers of $J$ and $I_\Delta^\vee$ (recall that they are the same) are given in the tables. The column {\em sizeRes} shows the size of the resolution of $J$ (and $I_\Delta^\vee$) computed as the total sum of the Betti numbers.
All experiments were run on an \texttt{M1} processor with \texttt{8GB RAM} using the computer algebra system \texttt{Macaulay2} (version \texttt{1.25.06}) \cite{M2}. All times reported are in seconds as retrieved by the \texttt{time} command in \texttt{Macaulay2}. \texttt{OOM} in a cell indicates \texttt{Out Of Memory}. The algorithm for computing Alexander dual ideals is the {\em slice algorithm} \cite{R09}.

\subsubsection{Powers of irrelevant ideals and ideals generated by powers of variables.}
Our first example consists of ideals of the form $I=\mk^k\subseteq\kb[x_1,\dots,x_n]$, minimally generated by all monomials of degree $k$ in $R$. Table \ref{table:equigenerated} shows the results for these ideals. In the table, $n$ indicates the number of variables and $k$ indicates the power to which we raise the ideal $\mk=\langle x_1,\dots,x_n\rangle$, and $n'=nk$ is the number of variables of the ring of $I_\Delta^\vee$. For these ideals, $\#(G(J^\vee))<\#(G(J))=\#(G(I_\Delta^\vee))<\#(G(I_\Delta))$, hence the total size (input plus output) for the dualization algorithm is bigger in the squarefree case. The computation times show this dependence on total size. The times for computing the Betti numbers are bigger in the squarefree case; however, the time ratio is smaller for the Betti number computation than for the Alexander dual, due to the fact that the minimal free resolution of these ideals is linear, i.e., it has the simplest possible form, and also the size of the resolution (sizeRes) is small with respect to the size of the minimal generating set.
\begin{small}
    
 \begin{table}[h]
 \centering
 \resizebox{\textwidth}{!}{
     \begin{tabular}{c|c|c|r|r|r|r|r|r|r|r}
        \toprule
          $n$&$k$&$n'$&$\#(J)$&$\#(J^\vee)$&$\#(I_\Delta)$&$t(J^\vee)$&$t(I_\Delta)$&$t(\beta(J))$&$t(\beta(I_\Delta^\vee))$&sizeRes\\
          \midrule
          $5$&$10$&$50$&1001&715&2002&0.0113&0.0902&0.1645&0.3128&13441\\
          $5$&$15$&$75$&3876&3060&11628&0.0454&0.7258&2.8479&5.7010&54911\\
          $5$&$20$&$100$&10626&8855&42504&0.1016&5.1972&21.2733&68.5737&154881\\
          $5$&$25$&$125$&23751&20475&118755&0.2158&31.9745&191.3712&612.9551&352351\\
          $5$&$30$&$150$&46376&40920&278256&0.5550&173.5453&1129.2431&2710.2512&696321\\
          $10$&$5$&$50$&2002&715&1001&0.0359&0.09228&34.8712&181.966&553983 \\
          $10$&$10$&$100$&92378&48620&92378&4.3685&69.8118&\texttt{OOM}&\texttt{OOM}&\texttt{OOM}\\
          \bottomrule
     \end{tabular}
     }
     \caption{Computation of the dual of the ideal $\mk^k$ and the dual of its polarization for several $n$ and $k$.}
      \label{table:equigenerated}
 \end{table}

\end{small}

 Our second example consists of ideals generated by powers of the variables, i.e., ideals of the form $J=\langle x_1^k,\dots x_n^k\rangle\subseteq \kb[x_1,\dots,x_n]$ for some $n$ and $k$. The dual ideal $J^\vee=\langle x_1\cdots x_n\rangle$ is generated by just one monomial, while the dual ideal $(I_\Delta^\vee)^\vee=I_\Delta$ is generated by $k^n$ monomials. This is one example in which the number of generators of the dual ideal is exponential with respect to the number of generators of the ideal (both $J$ and $I_\Delta^\vee$ are generated by $n$ monomials). Table \ref{table:powersVariables} shows the times and sizes for these ideals. In this case the size of the resolution is relatively bigger with respect to the number of generators of the ideal, but smaller than the dual of the polarization.

 \begin{table}[h]
  \centering
 \resizebox{\textwidth}{!}{
     \begin{tabular}{c|c|c|r|r|r|r|r|r|r|r}
          \toprule
          $n$&$k$&$n'$&$\#(J)$&$\#(J^\vee)$&$\#(I_\Delta)$&$t(J^\vee)$&$t(I_\Delta)$&$t(\beta(J))$&$t(\beta(I_\Delta^\vee))$&sizeRes\\
          \midrule
          $10$&$5$&$50$&10&1&9765625&0.0044&3.8607&0.0587&0.0416&1023\\ 
          $10$&$6$&$60$&10&1&60466176&0.0043&70.8422&0.0704&0.0414&1023\\ $10$&$7$&$70$&10&1&282475249&0.0043&\texttt{OOM}&0.0576&0.0460&1023\\ $10$&$8$&$80$&10&1&1073741824&0.0049&\texttt{OOM}&0.0708&0.0443&1023\\  
          \bottomrule
     \end{tabular}
     }
     \caption{Computation of the dual of $I=\langle x_1^k,\dots,x_n^k\rangle$ and the dual of its polarization for several $n$ and $k$.}
      \label{table:powersVariables}
 \end{table}

\subsubsection{Sums of ideals of the form $J_{k,n,m}=\langle \prod_{i\in\sigma}x_i^m\mid \sigma\subseteq[n], \vert\sigma\vert=k \rangle$}
Consider the ideal $J_{k,n,m}$ generated by all products of $k$ variables, each of them raised to the power $m$, i.e.,
\[
J_{k,n,m}=\langle \prod_{i\in\sigma}x_i^m\mid \sigma\subseteq[n], \vert\sigma\vert=k \rangle.
\]
Take now a sequence $(m_1,\dots, m_l)$ for some $l\leq n$ and consider the ideal $J=J_{1,n,m_1}+\cdots+J_{l,n,m_l}$.

For example, take $n=4$, $l=3$ and the sequence $(m_i\mid i=1,\dots,3)=(4,2,1)$, then
\begin{align*}
J=&\langle x_1^4,x_2^4,x_3^4,x_4^4\rangle+\langle x_1^2x_2^2,x_1^2x_3^2,x_1^2x_4^2,x_2^2x_3^2,x_2^2x_4^2,x_3^2x_4^2\rangle\\
&+\langle x_1x_2x_3,x_1x_2x_4,x_1x_3x_4,x_2x_3x_4 \rangle.
\end{align*}

Table \ref{table:sumsKN} shows the results for some ideals in this class. For each $n$ we consider the sequence $(2\lfloor\frac{n}{2}\rfloor-1, \dots,3,1)$ that consists of the first $\lfloor\frac{n}{2}\rfloor$ odd numbers. These ideals have a more complex resolution than the ones in the previous examples, in the sense that the generators of each module are not concentrated in one single degree, but instead spread along several ones. In these examples, the times for the computation of the dual ideal are small with respect to the computation of the Betti numbers of the ideal, unlike the previous example. The times for both computations are relatively high when compared to the previous classes of ideals considered.

 \begin{table}[h]
  \centering
 \resizebox{\textwidth}{!}{
     \begin{tabular}{c|c|r|r|r|r|r|r|r|r}
          \toprule
          $n$&$n'$&$\#(J)$&$\#(J^\vee)$&$\#(I_\Delta)$&$t(J^\vee)$&$t(I_\Delta)$&$t(\beta(J))$&$t(\beta(I_\Delta^\vee))$&sizeRes\\
          \midrule
          $6$&$30$&41&30&205&0.0020&0.0045&0.0226&0.0257&1131\\ 
          $7$&$35$&63&42&281&0.0054&0.0135&0.2566&0.3053&3277\\ 
          $8$&$56$&162&336&8113&0.0031&0.0946&0.9538&2.1037&37801\\   
          $9$&$63$&255&504&11959&0.0132&0.1458&4.8448&8.0140&116441\\   
          $10$&$90$&637&5040&470781&0.0343&5.6215&784.3065&4524.8317&1679683\\   

          \bottomrule
     \end{tabular}
     }
     \caption{Computation of the dual of ideals of the form $J=J_{1,n,2\lfloor\frac{n}{2}\rfloor-1}+\cdots+J_{\lfloor\frac{n}{2}\rfloor,n,1}$ and their polarizations.}
      \label{table:sumsKN}
 \end{table}

\subsubsection{Initial ideal of generic forms.}
Let $F_n\subseteq\mathbf{k}[x_1,\dots,x_n]$ be the ideal generated by $n$ generic quadratic forms in $n$ variables, and let $J_n=in_{\texttt{Lex}}(F_n)$ its initial ideal with respect to the \texttt{Degree Lex} ordering. Table \ref{table:generic} shows the times for the computations on these ideals. The times for the computation of the Betti numbers are much higher than the times for computing the dual. The difference in time for the computation of Betti numbers using the original ideal and its polarized one are bigger than in the previous cases. 
This example shows that for {\em complex} resolutions, it is worth depolarizing before computing Betti numbers.

 \begin{table}[h]
  \centering
 \resizebox{\textwidth}{!}{
     \begin{tabular}{c|c|r|r|r|r|r|r|r|r}
          \toprule
          $n$&$n'$&$\#(J)$&$\#(J^\vee)$&$\#(I_\Delta)$&$t(J^\vee)$&$t(I_\Delta)$&$t(\beta(J))$&$t(\beta(I_\Delta^\vee))$&sizeRes\\
          \midrule
          $8$&$35$&127&64&128&0.0030&0.0069&0.0239&0.1954&5461\\ 
          $9$&$44$&255&128&256&0.0061&0.0200&0.2598&0.7706&21845\\ 
          $10$&$54$&511&256&512&0.0087&0.0475&1.0138&5.0184&87381\\ 
          $11$&$65$&1023&512&1024&0.0186&0.0968&5.8176&46.0382&349525\\ 
          $12$&$77$&2047&1024&2048&0.0448&0.2098&39.7234&1260.9513&1398191\\ 
          \bottomrule
     \end{tabular}
     }
     \caption{Computation of the dual of initial ideals of $n$ generic forms and their polarizations.}
      \label{table:generic}
 \end{table}
\section{Conclusions and further work}\label{sec:conclusions}
We have translated the operations of polarization and depolarization from monomial ideals to simplicial complexes and demonstrated its advantages when performing demanding computations on simplicial complexes. In particular, we have applied these operations to computing Alexander dual of abstract simplicial complexes, showing that reducing the complex by depolarization can be a good tool for efficient algorithms.

On view of the computer experiments, one question that arises is how to describe or predict the ratio of the size, in terms of number of minimal monomial generators, of the dual of an ideal with respect to the dual of its polarization. Another interesting question is whether the ratio between the size of the dual of a monomial ideal and the size of the ideal itself, is related to some invariants of the ideal, such as regularity, size of the resolution, etc. (apart from the obvious relation to the ratio between the number of variables of the corresponding rings). Figure \ref{fig:ratios} shows that the behavior of the ratio between size of the dual ideal of the polarization and size of the original ideal versus the ratio between the number of variables of the corresponding rings is different among the different kinds of examples that we have seen. This is of course an important factor to estimate the difference in size of the duals, but as the examples show, it is not the only factor.

\begin{figure}[]
\centering
\begin{minipage}{.5\textwidth}
  \centering
  \includegraphics[width=\linewidth]{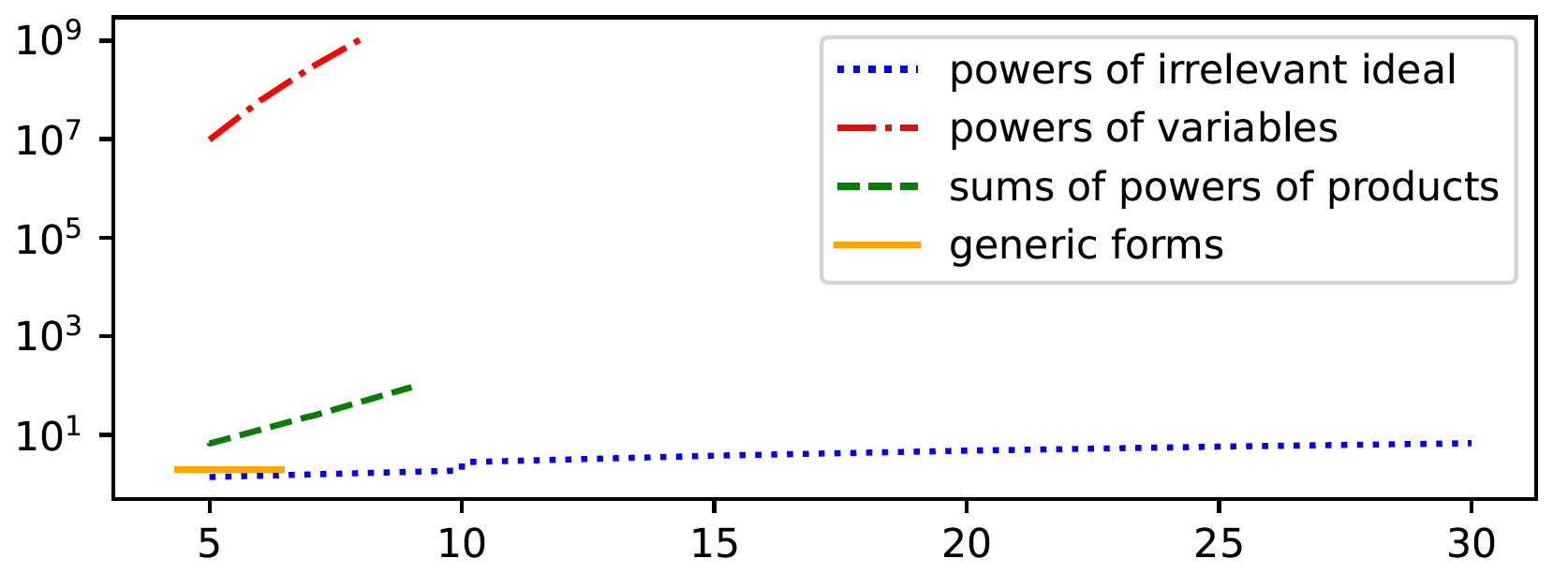}
  
\end{minipage}%
\begin{minipage}{.5\textwidth}
  \centering
  \includegraphics[width=\linewidth]{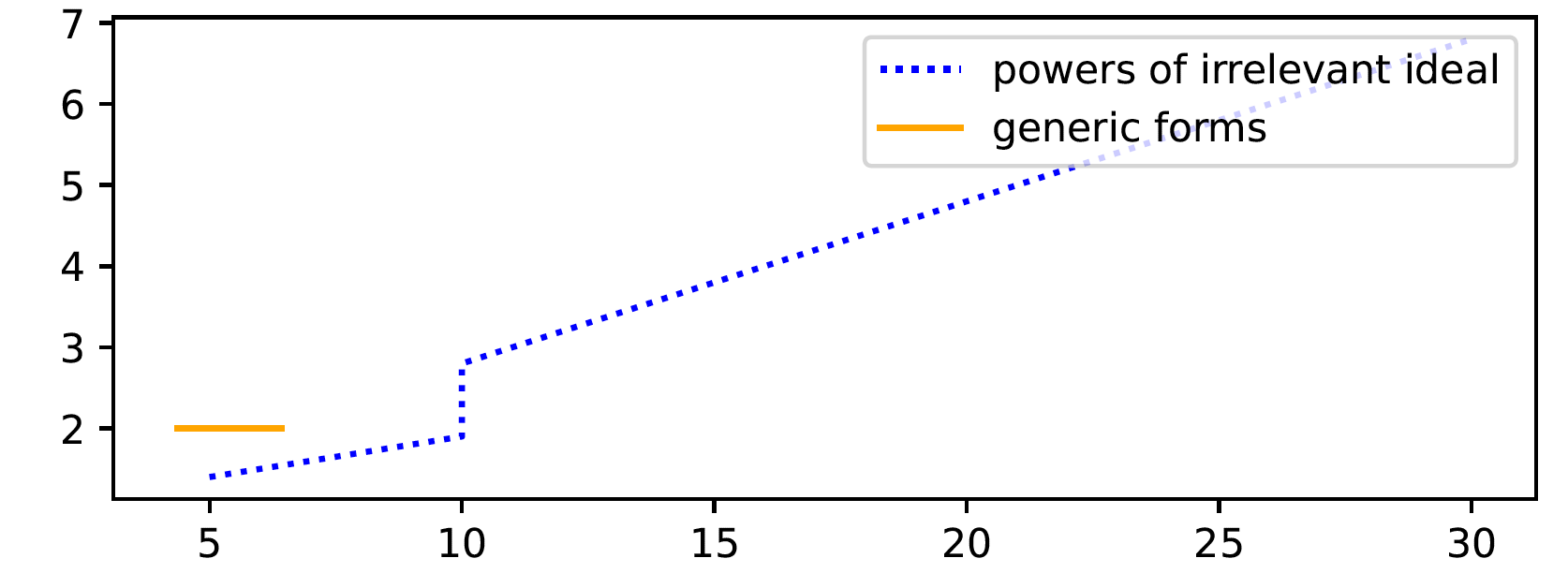}
  
\end{minipage}
\caption{Ratio between size of the dual ideal of the polarization and size of the original ideal versus ratio between the number of variables of the corresponding rings (right: detail of two of the example classes).}
\label{fig:ratios}
\end{figure}

\section*{Acknowledgements}
This work has been partially supported by grant INICIA2023/02 from Gobierno de La Rioja (Spain). The authors want to thank Rodrigo Iglesias, Laura Moreno and Julio Rubio for useful discussions and insights.

\appendix
\section{Shape of resolutions}\label{ap:Betti}
To illustrate the different structure of the resolutions of each class of examples, we show here an instance of each of them. Starting from the top left and reading clockwise, the following are the Betti diagrams corresponding to the resolutions of the four ideals
\begin{enumerate}
    \item[i)] $I=\langle x_1,\dots,x_5\rangle^{20}$, which has a linear resolution.
    \item[ii)] $I=\langle x_1^4,\dots,x_{10}^4\rangle$, which has a simple resolution in which each module is concentrated in one degree.
    \item[iii)] $J=J_{1,10,9}+\cdots+J_{5,10,1}$, whose resolution is big and with a complex structure of several layers.
    \item [iv)]$I$ is the initial ideal of $10$ generic quadratic forms with respect to the \texttt{Degree Lex} ordering, whose resolution is complete in the sense that each module has a nonzero Betti number for each possible degree smaller than or equal to the regularity of $I$.
\end{enumerate}

\begin{minipage}[t]{0.45\textwidth}
\centering
\resizebox{0.8\textwidth}{!}{
 $\begin{matrix}
      & 0 & 1 & 2 & 3 & 4 & 5\\
     \text{total:} & 1 & 10\,626 & 40\,480 & 57\,960 & 36\,960 & 8\,855\\
     0: & 1 & . & . & . & . & .\\
     1: & . & . & . & . & . & .\\
     2: & . & . & . & . & . & .\\
     3: & . & . & . & . & . & .\\
     4: & . & . & . & . & . & .\\
     5: & . & . & . & . & . & .\\
     6: & . & . & . & . & . & .\\
     7: & . & . & . & . & . & .\\
     8: & . & . & . & . & . & .\\
     9: & . & . & . & . & . & .\\
     10: & . & . & . & . & . & .\\
     11: & . & . & . & . & . & .\\
     12: & . & . & . & . & . & .\\
     13: & . & . & . & . & . & .\\
     14: & . & . & . & . & . & .\\
     15: & . & . & . & . & . & .\\
     16: & . & . & . & . & . & .\\
     17: & . & . & . & . & . & .\\
     18: & . & . & . & . & . & .\\
     19: & . & 10\,626 & 40\,480 & 57\,960 & 36\,960 & 8\,855
     \end{matrix}$
      }
\end{minipage}
\begin{minipage}[t]{0.55\textwidth}
\centering
 \resizebox{0.7\textwidth}{!}{
  $\begin{matrix}
       & 0 & 1 & 2 & 3 & 4 & 5 & 6 & 7 & 8 & 9 & 10\\
      \text{total:} & 1 & 10 & 45 & 120 & 210 & 252 & 210 & 120 & 45 & 10 & 1\\
      0: & 1 & . & . & . & . & . & . & . & . & . & .\\
      1: & . & . & . & . & . & . & . & . & . & . & .\\
      2: & . & . & . & . & . & . & . & . & . & . & .\\
      3: & . & 10 & . & . & . & . & . & . & . & . & .\\
      4: & . & . & . & . & . & . & . & . & . & . & .\\
      5: & . & . & . & . & . & . & . & . & . & . & .\\
      6: & . & . & 45 & . & . & . & . & . & . & . & .\\
      7: & . & . & . & . & . & . & . & . & . & . & .\\
      8: & . & . & . & . & . & . & . & . & . & . & .\\
      9: & . & . & . & 120 & . & . & . & . & . & . & .\\
      10: & . & . & . & . & . & . & . & . & . & . & .\\
      11: & . & . & . & . & . & . & . & . & . & . & .\\
      12: & . & . & . & . & 210 & . & . & . & . & . & .\\
      13: & . & . & . & . & . & . & . & . & . & . & .\\
      14: & . & . & . & . & . & . & . & . & . & . & .\\
      15: & . & . & . & . & . & 252 & . & . & . & . & .\\
      16: & . & . & . & . & . & . & . & . & . & . & .\\
      17: & . & . & . & . & . & . & . & . & . & . & .\\
      18: & . & . & . & . & . & . & 210 & . & . & . & .\\
      19: & . & . & . & . & . & . & . & . & . & . & .\\
      20: & . & . & . & . & . & . & . & . & . & . & .\\
      21: & . & . & . & . & . & . & . & 120 & . & . & .\\
      22: & . & . & . & . & . & . & . & . & . & . & .\\
      23: & . & . & . & . & . & . & . & . & . & . & .\\
      24: & . & . & . & . & . & . & . & . & 45 & . & .\\
      25: & . & . & . & . & . & . & . & . & . & . & .\\
      26: & . & . & . & . & . & . & . & . & . & . & .\\
      27: & . & . & . & . & . & . & . & . & . & 10 & .\\
      28: & . & . & . & . & . & . & . & . & . & . & .\\
      29: & . & . & . & . & . & . & . & . & . & . & .\\
      30: & . & . & . & . & . & . & . & . & . & . & 1
      \end{matrix}$
      }
\end{minipage}
\begin{minipage}{0.5\textwidth}
   
  \centering
 \resizebox{\textwidth}{!}{
 $\begin{matrix}
       & 0 & 1 & 2 & 3 & 4 & 5 & 6 & 7 & 8 & 9 & 10\\
      \text{total:} & 1 & 637 & 12\,360 & 73\,920 & 220\,815 & 392\,410 & 446\,316 & 330\,755 & 155\,310 & 42\,120 & 5\,040\\
      0: & 1 & . & . & . & . & . & . & . & . & . & .\\
      1: & . & . & . & . & . & . & . & . & . & . & .\\
      2: & . & . & . & . & . & . & . & . & . & . & .\\
      3: & . & . & . & . & . & . & . & . & . & . & .\\
      4: & . & 252 & 1\,050 & 1\,800 & 1\,575 & 700 & 126 & . & . & . & .\\
      5: & . & . & . & . & . & . & . & . & . & . & .\\
      6: & . & . & . & . & . & . & . & . & . & . & .\\
      7: & . & . & . & . & . & . & . & . & . & . & .\\
      8: & . & 10 & . & . & . & . & . & . & . & . & .\\
      9: & . & . & . & . & . & . & . & . & . & . & .\\
      10: & . & . & . & . & . & . & . & . & . & . & .\\
      11: & . & 210 & 2\,520 & 8\,190 & 12\,600 & 10\,350 & 4\,410 & 770 & . & . & .\\
      12: & . & . & . & . & . & . & . & . & . & . & .\\
      13: & . & 45 & . & . & . & . & . & . & . & . & .\\
      14: & . & 120 & 90 & . & . & . & . & . & . & . & .\\
      15: & . & . & 5\,040 & 17\,850 & 27\,720 & 22\,680 & 9\,600 & 1\,665 & . & . & .\\
      16: & . & . & 1\,680 & 15\,120 & 44\,100 & 63\,840 & 50\,400 & 20\,880 & 3\,570 & . & .\\
      17: & . & . & 720 & . & . & . & . & . & . & . & .\\
      18: & . & . & 1\,260 & 23\,400 & 69\,300 & 100\,800 & 79\,380 & 32\,760 & 5\,580 & . & .\\
      19: & . & . & . & 7\,560 & 60\,480 & 163\,800 & 226\,800 & 173\,880 & 70\,560 & 11\,880 & .\\
      20: & . & . & . & . & 5\,040 & 30\,240 & 75\,600 & 100\,800 & 75\,600 & 30\,240 & 5\,040
      \end{matrix}$
      }
\end{minipage}
\begin{minipage}{0.5\textwidth}
   
  \centering
 \resizebox{\textwidth}{!}{
$\begin{matrix}
       & 0 & 1 & 2 & 3 & 4 & 5 & 6 & 7 & 8 & 9\\
      \text{total:} & 1 & 511 & 3\,586 & 11\,260 & 20\,488 & 23\,536 & 17\,440 & 8\,128 & 2\,176 & 256\\
      0: & 1 & . & . & . & . & . & . & . & . & .\\
      1: & . & 9 & 36 & 84 & 126 & 126 & 84 & 36 & 9 & 1\\
      2: & . & 36 & 204 & 546 & 882 & 924 & 636 & 279 & 71 & 8\\
      3: & . & 84 & 546 & 1\,596 & 2\,730 & 2\,976 & 2\,109 & 946 & 245 & 28\\
      4: & . & 126 & 882 & 2\,730 & 4\,878 & 5\,499 & 4\,001 & 1\,833 & 483 & 56\\
      5: & . & 126 & 924 & 2\,976 & 5\,499 & 6\,376 & 4\,750 & 2\,220 & 595 & 70\\
      6: & . & 84 & 636 & 2\,109 & 4\,001 & 4\,750 & 3\,614 & 1\,721 & 469 & 56\\
      7: & . & 36 & 279 & 946 & 1\,833 & 2\,220 & 1\,721 & 834 & 231 & 28\\
      8: & . & 9 & 71 & 245 & 483 & 595 & 469 & 231 & 65 & 8\\
      9: & . & 1 & 8 & 28 & 56 & 70 & 56 & 28 & 8 & 1
      \end{matrix}$
      }
\end{minipage}
\end{document}